\documentclass[10pt]{article}

\usepackage[a4paper, total={6in, 9in}]{geometry}
\usepackage{color}
\usepackage{graphicx}
\usepackage{latexsym}
\usepackage{amsfonts,amsmath,amssymb,amsthm}

\newtheorem{theorem}{Theorem}

\newcommand{\ord}{\mathrm{ord}}

\begin{document}

\author{Istv\'an Mez\H{o}\thanks{The research of Istv\'an Mez\H{o} was supported by the Scientific Research Foundation of Nanjing University of Information Science \& Technology, the Startup Foundation for Introducing Talent of NUIST. Project no.: S8113062001, and the National Natural Science Foundation for China. Grant no. 11501299.}\\Department of Mathematics\\Nanjing University of Information Science and Technology\\
Nanjing, 210044, P. R. China}

\title{The $p$-adic Lambert $W$ function}
\maketitle

\begin{abstract}
In this paper we introduce the $p$-adic analogue of the Lambert $W$ function, and study its main properties.
\end{abstract}

\section{Introduction}

The classical Lambert $W$ function is a special function gaining more and more interest from mathematicians and physicists. It can be defined on the whole complex plane by the transcendental equation
\begin{equation}
W(z)e^{W(z)}=z.\label{W_def}
\end{equation}
As this equation has infinitely many solutions (except when $z=0$), the $W$ function has infinitely many branches. See the basic source \cite{W} for the main properties of $W$, and \cite{W,Houari,Mezo,Valluri} for references on the many applications $W$ has.

In the present paper we define the $p$-adic analogue of the Lambert $W$ function, and study its properties. The proof of the results are slightly less trivial than in the classical and well known case of the $p$-adic exponential function $\exp_p(x)$ and its inverse $\log_p(x)$. Thus the study of this new function has demonstrative power.

Let $\Omega_p$ be the algebraically and topologically closed (and spherically complete) $p$-adic field\footnote{We will not need this generality, however. Our analysis works even on $\mathbb{Q}_p$.} for a prime $p$. Based on \eqref{W_def}, we define $W_p(x)$ to be a function on (a part of) $\Omega_p$ such that
\[W_p(x)\exp_p(W_p(x))=x.\]

The theory of $p$-adic functions dictates that the inversion of $xe^x$ and $x\exp_p(x)$ results in the same Taylor series, only the radius of convergence changes when we switch from the standard topology generated by $|\cdot|$ to the $p$-adic topology based on $|\cdot|_p$. Thus, we have that $W_p(x)$ is represented by the very same Taylor series as the classical Lambert $W$ function:
\begin{equation}
W_p(x)=\sum_{n=1}^\infty\frac{(-n)^{n-1}}{n!}x^n.\label{Wp}
\end{equation}

First we study the basic mapping properties of $W_p$, then we prove that it cannot be represented as a uniform limit of rational $p$-adic functions.

\section{The basic mapping properties of $W_p$}

We prove the following theorem.

\begin{theorem}The series \eqref{Wp} defining $W_p$ is convergent whenever $x\in\Omega_p$ such that $|x|_p<p^{-\frac{1}{p-1}}=:r_p$, and it is divergent elsewhere. Moreover, it is true that
\[|W_p(x)|_p=|x|_p\quad(|x|_p<r_p).\]
Thus $W_p$ has no zeros on its domain of definition except $x=0$.

For the growth modulus we have that
\[M_r(W_p)\stackrel{\mathrm{def}}{=}\max_{n\ge1}\left|\frac{(-n)^{n-1}}{n!}\right|_pr^n=r\quad(0<r<r_p),\]
and $W_p$ has no critical radius\footnote{A critical radius of a function $f(x)=\sum_{n\ge0}a_nx^n$ is a positive real number $r$ such that $|a_n|_p\cdot r^n=|a_m|_p\cdot r^m$ for some $n,m$ such that $n\neq m$. Critical radii are important in the study of zeros of a $p$-adic function, but, as we see, $W_r$ has no non-trivial zeros. More on the critical radii can be read in \cite[6.1-6.2]{Robert}}.
\end{theorem}

\begin{proof}That the radius of convergence for \eqref{Wp} is $r_p$ can be seen easily: for $n$ which is not a multiple of $p$, $\left|\frac{(-n)^{n-1}}{n!}\right|_p=\left|\frac{1}{n!}\right|_p$, which shows that for such $n$ \eqref{Wp} contains a partial series of the $p$-adic exponential
\[\exp_p(x)=\sum_{n\ge0}\frac{1}{n!}x^n,\]
and it is well known that this series' radius of convergence is $r_p$ (\cite[p. 251]{Robert}, \cite[p. 79]{Koblitz}), and on the radius $|x|_p=r_p$ the series in question is divergent.

That $|W_p(x)|_p=|x|_p$ can be seen as follows. We take the Taylor series, and we show that the first term, $x$, dominates the rest in absolute value. To this end, let us fix $n>1$, fix an $x$ such that $|x|_p<r_p$, and carry out the estimation
\[\left|\frac{(-n)^{n-1}}{n!}x^n\right|_p=|x|_p\left|\frac{(-n)^{n-1}x^{n-1}}{n!}\right|_p\le|x|_p\left|\frac{x^{n-1}}{n!}\right|_p\le|x|_p\frac{|x^{n-1}|_p}{r_p^{n-1}}=|x|_p\left(\frac{|x|_p}{r_p}\right)^{n-1}<|x|_p.\]
Here we used the simple fact that $|n!|_p\ge r_p^{n-1}$:
\[|n!|_p=p^{-\frac{n-S_n}{p-1}}\ge p^{-\frac{n-1}{p-1}}=r_p^{n-1}.\]
($S_n$ is the sum of the digits of $n$ in base $p$.)

The statement on the growth modulus is proven by considering the following steps.
\[M_r(W_p)=\max_{n\ge1}\left|\frac{(-n)^{n-1}}{n!}\right|_pr^n=\max_{n\ge1}p^{-(n-1)\ord_p(-n)+\ord_p(n!)}\cdot r^n=\]
\[\max_{n\ge1}p^{-(n-1)\ord_p(n)+\frac{n-S_n}{p-1}}\cdot r^n=\max_{n\ge1}\left(p^{\frac{1}{p-1}}r\right)^np^{-\frac{S_n}{p-1}}p^{-(n-1)\ord_p(n)}.\]
Now it can be seen that the maximum is attained when $n=1$, so
\[M_r(W_p)=p^{\frac{1}{p-1}}\cdot r\cdot p^{-\frac{1}{p-1}}=r.\]

As critical radii characterize the absolute value of the zeros of an analytic function, and $w_p$ has no zeros of positive magnitude, it follows that $W_p$ has no critical radii.

All the statements of the theorem are proved.
\end{proof}

\section{$W_p(\pi x)$ is not an analytic element}

An analytic element is a function that can be represented as a uniform limit of rational functions \cite[p. 342]{Robert}. The Christol-Robba theorem provides a condition which often helps to decide whether a given $p$-adic analytic function is an analytic element or not. Let $f=\sum_{n\ge0}a_nx^n$ be a formal power series with bounded coefficients. Define $p_\nu=p^\nu(p^\nu-1)$ for $\nu\ge1$. Then $f$ defines an analytic element on
\[\mathbf{M}_p=\{x\in\Omega_p\mid|x|_p<1\}\]
if and only if the following condition holds. For each $\varepsilon>0$ there exists $\nu$ and $N$ positive integers such that
\[|a_{n+p_\nu}-a_n|_p<\varepsilon\]
for all $n\ge N$. This condition is called the Christol-Robba condition.

Since
\[\left|\frac{(-n)^{n-1}}{n!}p^{\frac{n}{p-1}}\right|_p\le\left|\frac{1}{n!}p^{\frac{n}{p-1}}\right|_p=p^{-\frac{n}{p-1}+\frac{n-S_n}{p-1}}=p^{\frac{-S_n}{p-1}}\le p^{-\frac{1}{p-1}},\]
the Taylor coefficients of $W_p(\pi x)$ are bounded, and the Cristol-Robba theorem can be applied to $W_p(\pi x):\mathbf{M}_p\to\mathbf{M}_p$. Here $\pi$ is a $p$-adic number used to rescale the arguments such that $W_p$ can be defined on the open unit disk $\mathbf{M}_p$ of $\Omega_p$. It can be choosen to be any of the roots of the equation $x^{p-1}-p=0$.

We are now ready to prove the following statement.

\begin{theorem}The function $W_p(\pi x):\mathbf{M}_p\to\mathbf{M}_p$ is not an analytic element.
\end{theorem}

\begin{proof}Let us fix an $1>\varepsilon>0$, and a $\nu$ as it is given in the CR-theorem. If we prove that there are infinitely many $n$ such that $|a_{n+p_\nu}-a_n|_p\not\le\varepsilon$ with $a_n=\frac{(-n)^{n-1}}{n!}p^{\frac{n}{p-1}}$, then we will be done. In fact, we will prove that this absolute value is always bounded from below. To this end, we fix $n=p^\alpha k+1$ such that $\alpha>2\nu$. Then note that
\begin{align}
\ord_p(n)&=\ord_p(n+p_\nu)=0,\label{nobs1}\\
\intertext{and}
S_{n+p_\nu}&=S_n+S_{p_\nu}.\label{nobs2}
\end{align}
Now consider
\begin{equation}
|a_{n+p_\nu}-a_n|_p=\left|p^{\frac{n}{p-1}}\frac{(-n)^{n-1}}{n!}\right|_p\left|(-1)^{p_\nu}\left(\frac{n+p_\nu}{n}\right)^{n-1}(n+p_\nu)^{p_\nu}p^{\frac{p_\nu}{p-1}}\frac{n!}{(n+p_\nu)!}-1\right|_p.\label{CRcond}
\end{equation}
Since the chosen $n$s are not divisible by $p$, we have that
\[\left|p^{\frac{n}{p-1}}\frac{(-n)^{n-1}}{n!}\right|_p=\left|\frac{p^{\frac{n}{p-1}}}{n!}\right|_p=p^{-\frac{S_n}{p-1}}\ge p^{-\frac{1}{p-1}}.\]
Moreover, the $p$-adic order of the first expression in the second absolute value of \eqref{CRcond} is equal to
\[(n-1)\ord_p\left(\frac{n+p_\nu}{n}\right)+p_\nu\ord_p(n+p\nu)+\frac{p_\nu}{p-1}+\frac{1}{p-1}\left(n-S_n-(n+p_\nu)+S_{n+p\nu}\right).\]
By the observations \eqref{nobs1}-\eqref{nobs2} this simplifies to
\[\frac{1}{p-1}S_{p_\nu}=\nu.\]
This last equality is trivial:
\[S_{p_\nu}=S_{p^\nu(p^\nu-1)}=S_{p^\nu-1}=S_{(p-1)p^0+(p-1)p^1+\cdots+(p-1)p^{\nu-1}}=\nu(p-1).\]
We therefore have that the first term in the second absolute value of \eqref{CRcond} is divisible by $p^\nu$, so the $p$-adic number in the absolute value is a unit.

Collecting all the information we get that
\[|a_{n+p_\nu}-a_n|_p\ge p^{-\frac{1}{p-1}},\]
thus it is indeed bounded from below, as we claimed.
\end{proof}

\section*{Acknowledgement}

The author is grateful to Huang XuePing for raising the question of the study of the $p$-adic Lambert function.

\end{document}